\documentclass[oneside,11pt]{amsart}
\usepackage{amsmath, amsfonts,amsthm,times,graphics}
\usepackage{mathrsfs}
\usepackage[usenames]{color}
\usepackage[active]{srcltx}
 \makeatletter
\renewcommand*\subjclass[2][2000]{%
  \def\@subjclass{#2}%
  \@ifundefined{subjclassname@#1}{%
    \ClassWarning{\@classname}{Unknown edition (#1) of Mathematics
      Subject Classification; using '1991'.}%
  }{%
    \@xp\let\@xp\subjclassname\csname subjclassname@#1\endcsname
  }%
}
 \makeatother

\newtheorem*{ThmA}{Theorem A}
\newtheorem*{ThmB}{Theorem B}
\newtheorem*{ThmC}{Theorem C}
\newtheorem*{ThmD}{Theorem D}

\newtheorem*{ThmF}{Theorem F}

\newtheorem*{LemE}{Lemma E}

\newtheorem*{LemG}{Lemma G}

\newtheorem{Thm}{Theorem}[section]

\newtheorem{Lem}[Thm]{Lemma}

\theoremstyle{definition}

\theoremstyle{remark}
\newtheorem{Rem}[Thm]{\upshape\bfseries Remark}
\newtheorem{Ques}[Thm]{Question}

\numberwithin{equation}{section}

\newcommand{\ee}{\mathrm{e}}

\theoremstyle{definition}

\def\be{\begin{equation}}
\def\ee{\end{equation}}

\newcommand{\ben}{\begin{enumerate}}
\newcommand{\een}{\end{enumerate}}

\newcommand{\br}{\begin{rem}}
\newcommand{\er}{\end{rem}}
\newcommand{\brs}{\begin{rems}}
\newcommand{\ers}{\end{rems}}
\newcommand{\bo}{\begin{obser}}
\newcommand{\eo}{\end{obser}}
\newcommand{\bos}{\begin{obsers}}
\newcommand{\eos}{\end{obsers}}
\newcommand{\bpf}{\begin{pf}}
\newcommand{\epf}{\end{pf}}
\newcommand{\ba}{\begin{array}}
\newcommand{\ea}{\end{array}}
\newcommand{\beq}{\begin{eqnarray}}
\newcommand{\beqq}{\begin{eqnarray*}}
\newcommand{\eeq}{\end{eqnarray}}
\newcommand{\eeqq}{\end{eqnarray*}}

\numberwithin{equation}{section}

\newcounter{minutes}
\divide\time by 60
\newcounter{hours}
\multiply\time by 60 \addtocounter{minutes}{-\time}
\begin{document}
\title{Composition
operators on Bloch and Hardy type spaces}

\author[Shaolin Chen, Hidetaka Hamada, and Jian-Feng Zhu]{Shaolin Chen, Hidetaka Hamada, and Jian-Feng Zhu}

\address{S. L.  Chen, College of Mathematics and
Statistics, Hengyang Normal University, Hengyang, Hunan 421002,
People's Republic of China; Hunan Provincial Key Laboratory of
Intelligent Information Processing and Application,  421002,
People's Republic of China.} \email{mathechen@126.com}

\address{H. Hamada, Faculty of Science and Engineering, Kyushu Sangyo University,
3-1 Matsukadai 2-Chome, Higashi-ku, Fukuoka 813-8503, Japan.}
\email{ h.hamada@ip.kyusan-u.ac.jp}

\address{J.-F. Zhu, School of Mathematical Sciences,
Huaqiao University,
Quanzhou-362021, China}

\email{flandy@hqu.edu.cn}


\maketitle

\def\thefootnote{}
\makeatletter\def\thefootnote{\@arabic\c@footnote}\makeatother

\begin{abstract}
The main purpose of this paper is to discuss  Hardy type spaces,   Bloch type spaces and the composition operators of complex-valued
harmonic functions.
We first  establish a sharp estimate
of the Lipschitz continuity of complex-valued harmonic functions in Bloch type spaces with respect to   the pseudo-hyperbolic metric,
which gives an answer to an open problem.
Then  some classes of composition operators on Bloch and Hardy type spaces will be investigated. The obtained results  improve and extend
some corresponding known results.
\end{abstract}

\maketitle \pagestyle{myheadings} \markboth{ S. L. Chen, H. Hamada and J.-F. Zhu}{Composition
operators on Bloch and Hardy type spaces}

\maketitle
\tableofcontents

\section{ Introduction}\label{sec1}
Recently, the characterizations of composition operators on Bloch
and Hardy type spaces have been attracted much attention of many
mathematicians (one can see the references
\cite{AD,CPR,GYZ,GPP,GPPJ,HO-0,HO,Kw,Sha,Z2,Z1} for more details).
This paper is mainly motivated by the results given by Chen  et
al.\cite{CPR},
Ghatage et al. \cite{GYZ},
Huang et al. \cite{HRZ},   and
Madigan \cite{Ma}.
In
order to state our main results, we need to recall some basic
definitions and some results which motivate the present work.


Let $\mathbb{D}=\{z\in \mathbb{C} :~|z|<1\}$ be the unit disk, and let $\mathbb{T}=:\partial\mathbb{D}$ be the unit circle.

For $z=x+iy\in\mathbb{C}$, the complex formal differential operators
are defined by
$$\frac{\partial}{\partial z}=\frac{1}{2}\left(\frac{\partial}{\partial x}-i\frac{\partial}{\partial y}\right)\ \ \ \mbox{and}\ \ \
\frac{\partial}{\partial \bar{z}}=\frac{1}{2}\left(\frac{\partial}{\partial x}+i\frac{\partial}{\partial y}\right).$$
For
$\alpha\in[0,2\pi]$, the  directional derivative of  a
complex-valued harmonic function $f$ at $z\in\mathbb{D}$ is defined
by


$$\partial_{\alpha}f(z)=\lim_{\rho\rightarrow0^{+}}\frac{f(z+\rho e^{i\alpha})-f(z)}{\rho}=f_{z}(z)e^{i\alpha}
+f_{\overline{z}}(z)e^{-i\alpha},$$ where $f_{z}=:\partial
f/\partial z,$ $f_{\overline{z}}=:\partial f/\partial \overline{z}$
and $\rho$ is a positive real number such that $z+\rho
e^{i\alpha}\in\mathbb{D}$. Then

$$\Lambda_{f}(z)=:\max\{|\partial_{\alpha}f(z)|:\; \alpha\in[0,2\pi]\}=|f_{z}(z)|+|f_{\overline{z}}(z)|.
$$



 It is well-known that
every complex-valued harmonic function $f$ defined in a simply
connected domain $\Omega$ admits the canonical decomposition $f = h
+ \overline{g}$, where $h$ and $g$ are analytic in $\Omega$ with $g(0)=0$.


Denote by  $\mathscr{A}$ and $\mathscr{H}$  the set of all analytic functions of $\mathbb{D}$ into $\mathbb{C}$
 and all complex-valued harmonic functions of $\mathbb{D}$ into $\mathbb{C}$, respectively.

Throughout of this paper,
 we use the symbol $C$ to denote the various positive
constants, whose value may change from one occurrence to another.

\section{Preliminaries and main results}\label{sec2}
\subsection{Hardy type spaces}
For $p\in(0,+\infty]$, the  generalized Hardy space
$\mathscr{H}^{p}_{G}(\mathbb{D})$ consists of all those measurable functions
$f:\ \mathbb{D}\rightarrow\mathbb{C}$ such that, for $0<p<+\infty$,
$$\|f\|_{p}=:\sup_{0<r<1}M_{p}(r,f)<+\infty,$$
and, for $p=+\infty$,
$$\|f\|_{p}=:\sup_{z\in\mathbb{D}}|f(z)|<+\infty,$$
where
$$M_{p}(r,f)=\left(\frac{1}{2\pi}\int_{0}^{2\pi}|f(re^{i\theta})|^{p}\,d\theta\right)^{\frac{1}{p}}.
$$
Let
$\mathscr{H}_{H}^{p}(\mathbb{D})=\mathscr{H}^{p}_{G}(\mathbb{D})\cap\mathscr{H}$
be the harmonic  Hardy space. The classical  Hardy space
$\mathscr{H}^{p}(\mathbb{D})$, that is, all the elements are
analytic, is a subspace of $\mathscr{H}_{H}^{p}(\mathbb{D})$ (see
\cite{CPR,CPW-2021,Du1,Du,K-2019,Zhu}).

\subsection{Bloch type spaces}

A continuous increasing function $\omega:[0,+\infty)\rightarrow[0,+\infty)$ with $\omega(0)=0$ is called a majorant if
$\omega(t)/t$ is non-increasing for $t>0$ (see \cite{Dy1,Dy2,P}).
For $\alpha>0$, $\beta\in\mathbb{R}$, $1\leq\,p\leq+\infty$ and a  majorant $\omega$, we use $\mathscr{B}_{H_{\omega,p}}^{\alpha,\beta}$ to denote
the harmonic Bloch type space consisting of all $f\in\mathscr{H}$ with the norm
$$\|f\|_{\mathscr{B}_{H_{\omega,p}}^{\alpha,\beta}}=:|f(0)|+\sup_{z\in\mathbb{D}}\mathscr{B}_{\omega,f}^{\alpha,\beta,p}(z)<+\infty,$$
where  $$\mathscr{B}_{\omega,f}^{\alpha,\beta,p}(z)=M_{p}(|z|,\Lambda_{f}(z))\omega\left(\left(1-|z|^{2}\right)^{\alpha}
\left(\log \frac{e}{1-|z|^{2}}\right)^{\beta}\right)$$ for $1\leq\, p<+\infty$,
and $$\mathscr{B}_{\omega,f}^{\alpha,\beta,p}(z)=\Lambda_{f}(z)\omega\left(\left(1-|z|^{2}\right)^{\alpha}
\left(\log \frac{e}{1-|z|^{2}}\right)^{\beta}\right)$$ for $p=+\infty$.
In particular, if $p=+\infty$, then we let $\mathscr{B}_{H_{\omega}}^{\alpha,\beta}=:\mathscr{B}_{H_{\omega,p}}^{\alpha,\beta}$,
$\|f\|_{\mathscr{B}_{H_{\omega}}^{\alpha,\beta}}=:\|f\|_{\mathscr{B}_{H_{\omega,p}}^{\alpha,\beta}}$ and
$\mathscr{B}_{\omega,f}^{\alpha,\beta}(z)=:\mathscr{B}_{\omega,f}^{\alpha,\beta,p}(z)$ for $z\in\mathbb{D}$.
Set
$$
\|f\|_{\mathscr{B}_{H_{\omega},s}^{\alpha,\beta}}
=\sup_{z\in\mathbb{D}}\mathscr{B}_{\omega,f}^{\alpha,\beta}(z)
$$
be the semi-norm.
 For the characterizations of $\mathscr{B}_{H_{\omega,p}}^{\alpha,\beta}$, we refer to the reference \cite{CPR}.
  In this paper, we mainly focus on the case $p=+\infty$. The case of $p\in[1,+\infty)$ is probably of independent interest.

In particular, if $\omega(t)=t$ and $f\in\mathscr{B}_{H_{\omega}}^{\alpha,0}$ ($f\in\mathscr{B}_{H_{\omega}}^{1,0}$ resp.), then
we call $f$ the harmonic $\alpha$-Bloch mapping (the harmonic Bloch mapping resp. ). Moreover,
 if $\omega(t)=t$ and $f\in\mathscr{B}_{H_{\omega}}^{1,0}\cap\mathscr{A}$,
 then we call $f$ the analytic Bloch mapping.

For $z,~w\in\mathbb{D}$, the pseudo-hyperbolic
metric is defined as $$\rho(z,w)=\left|\frac{z-w}{1-\overline{w}z}\right|.$$
Colonna \cite{Co-1989} proved that  if $f\in \mathscr{H}$ satisfies
\be\label{eq-1}\sup_{z,w\in\mathbb{D},z\neq w}\frac{|f(z)-f(w)|}{\sigma(z,w)}<+\infty,\ee
then $f\in\mathscr{B}_{H_{\omega}}^{1,0}$,
 where $\omega(t)=t$ and $$\sigma(z,w)=\frac{1}{2}\log\left(\frac{1+\rho(z,w)}{1-\rho(z,w)}\right)
=\mbox{arctanh}(\rho(z,w))$$ is the hyperbolic distance between $z$
and $w$ in $\mathbb{D}$. It is easy to conclude the converse, so
$f\in \mathscr{H}$ satisfies $(\ref{eq-1})$ if and only if
$f\in\mathscr{B}_{H_{\omega}}^{1,0}$.

For $\omega(t)=t$ and an analytic Bloch mapping $f$,  Ghatage et al. \cite{GYZ} proved that $\mathscr{B}_{\omega,f}^{1,0}(z)$
is Lipschitz continuous with respect to the pseudo-hyperbolic
metric, which is given as follows.

\begin{ThmA}{\rm (\cite[Theorem 1]{GYZ})}\label{Thm-1}
Let $\omega(t)=t$ and $f$ be an analytic Bloch mapping. Then, for all
$z_{1},~z_{2}\in\mathbb{D}$,

$$\left|\mathscr{B}_{\omega,f}^{1,0}(z_{1})-\mathscr{B}_{\omega,f}^{1,0}(z_{2})\right|\leq3.31\|f\|_{\mathscr{B}_{H_{\omega}}^{1,0}}\rho(z_{1},z_{2}).$$

\end{ThmA}

In \cite{HO-0}, Hosokawa and Ohno showed that if $f$ is an analytic
Bloch mapping, then,  for all $z_{1},~z_{2}\in\mathbb{D}$,

\be\label{eq-2}\left|\mathscr{B}_{\omega,f}^{1,0}(z_{1})-\mathscr{B}_{\omega,f}^{1,0}(z_{2})\right|\leq20\|f\|_{\mathscr{B}_{H_{\omega}}^{1,0}}\rho(z_{1},z_{2}),\ee
where $\omega(t)=t$.
They used (\ref{eq-2}) to discuss the composition operators on the
Bloch spaces (see \cite{HO-0,HO}). On some related discussions of
(\ref{eq-2}), we refer to \cite{CK,H,X}.

It is natural to ask whether
Theorem A also holds for harmonic Bloch mappings?

The following  explanations show that the answer to this question is positive.
 For $\omega(t)=t$, let
$f=h+\overline{g}$ be a harmonic Bloch mapping, where $h$ and $g$
are analytic in $\mathbb{D}$. Obviously, $h$ and $g$ are analytic
Bloch mappings.  Then, by Theorem A, we see that there is
a positive constant $C$ such that

\beq\label{0.93}
\nonumber\left|\mathscr{B}_{\omega,f}^{1,0}(z_{1})-\mathscr{B}_{\omega,f}^{1,0}(z_{2})\right|&\leq&\left|\mathscr{B}_{\omega,h}^{1,0}(z_{1})-\mathscr{B}_{\omega,h}^{1,0}(z_{2})\right|
+\left|\mathscr{B}_{\omega,g}^{1,0}(z_{1})-\mathscr{B}_{\omega,g}^{1,0}(z_{2})\right|\\
\nonumber
&\leq&C\|h\|_{\mathscr{B}_{H_{\omega}}^{1,0}}\rho(z_{1},z_{2})+C\|g\|_{\mathscr{B}_{H_{\omega}}^{1,0}}\rho(z_{1},z_{2})\\
 &\leq&2C\|f\|_{\mathscr{B}_{H_{\omega}}^{1,0}}\rho(z_{1},z_{2}), \eeq
which implies that $\mathscr{B}_{\omega,f}^{1,0}(z)$ is also Lipschitz
continuous with respect to the pseudo-hyperbolic metric.

In
\cite[Theorem 1.1]{HRZ}, Huang et al. showed that the Lipschitz constant in
(\ref{0.93}) is $C= 3\sqrt{3}/2$, which is as follows.

\[
\left|\mathscr{B}_{\omega,f}^{1,0}(z_{1})-\mathscr{B}_{\omega,f}^{1,0}(z_{2})\right|
\leq3\sqrt{3}\|f\|_{\mathscr{B}_{H_{\omega}}^{1,0}}\rho(z_{1},z_{2}).
\]

Furthermore, Huang et al. \cite[Open problems]{HRZ} raised the following open problem.

\begin{Ques}\label{qes-2.8}
 Let
 $$\mathcal{C}=\sup_{z_{1}\neq z_{2},~z_{1},z_{2}\in\mathbb{D},~\|f\|_{\mathscr{B}_{H_{\omega}}^{1,0}}=1}
 \frac{\left|\mathscr{B}_{\omega,f}^{1,0}(z_{1})-\mathscr{B}_{\omega,f}^{1,0}(z_{2})\right|}{\rho(z_{1},z_{2})},$$
 where $\omega(t)=t$ and $f$ is a harmonic Bloch mapping.
What is the  value of $\mathcal{C}$?
\end{Ques}

 In the following, we give an answer to Question \ref{qes-2.8}.


\begin{Thm}\label{thm-1CH}
Let $\omega(t)=t$ and $f$ be a harmonic Bloch mapping. Then, for all $z_{1},~z_{2}\in\mathbb{D}$,

\be\label{eq-3}\left|\mathscr{B}_{\omega,f}^{1,0}(z_{1})-\mathscr{B}_{\omega,f}^{1,0}(z_{2})\right|\leq
\frac{3\sqrt{3}}{2}\|f\|_{\mathscr{B}_{H_{\omega},s}^{1,0}}\rho(z_{1},z_{2}).\ee
Moreover, the constant $3\sqrt{3}/2$ in {\rm (\ref{eq-3})} is sharp.
\end{Thm}

\subsection{Composition operators}

Given an analytic self mapping $\phi$ of the unit disk $\mathbb{D}$, the composition operator
$C_{\phi}:~\mathscr{H}\rightarrow\mathscr{H}$ is defined by
$$C_{\phi}(f)=f\circ\phi,$$
where $f\in\mathscr{H}$.

In \cite{Sha}, Shapiro obtained a complete characterization of compact composition operators on $\mathscr{H}^{2}(\mathbb{D})$,
with a number of interesting consequences for peak sets, essential norm of composition operators, and so on. Recently,
the studies of composition operators on analytic function spaces have been attracted
much attention of many mathematicians (see references \cite{AD,CPR,GYZ,GPPJ,HO-0,HO,Kw,Z2}).
In particular, the characterizations of composition operators of the
Bloch spaces to the  Hardy spaces were investigated by Abakumov and Doubtsov \cite{AD},
Hosokawa and Ohno \cite{HO}, and Kwon \cite{Kw}. However, there are few literatures on the
theory of composition operators of complex-valued harmonic functions. In the following, we will discuss  the characterizations of composition operators on
complex-valued harmonic functions.

As an application of Theorem A, Ghatage et al. \cite{GYZ} obtained the following result.

\begin{ThmB}{\rm (\cite[Theorem 2]{GYZ})}\label{Thm-GHZ-2}
Let $\omega(t)=t$ and $\phi$ be an analytic self mapping  of the unit disk $\mathbb{D}$. If for some constants
$r\in(0,1/4)$, and $\epsilon>0$, for each $w\in\mathbb{D}$, there is a point $z_{w}\in\mathbb{D}$ such that
$$\rho(\phi(z_{w}),w)<r,~\mbox{and}~\frac{1-|z_{w}|^{2}}{1-|\phi(z_{w})|^{2}}|\phi'(z_{w})|>\epsilon,$$
then, there is a constant $C>0$ depending only on $r$ and $\epsilon$
such that $$\|C_{\phi}(f)\|_{\mathscr{B}_{H_{\omega}}^{1,0}}\geq C\|f\|_{\mathscr{B}_{H_{\omega}}^{1,0}}$$
for all $f\in\mathscr{B}_{H_{\omega}}^{1,0}\cap\mathscr{A}$.
\end{ThmB}

In the following, by using Theorem \ref{thm-1CH}, we  extend Theorem
B to complex-valued harmonic functions, and also show that
``$r\in(0,1/4)$" in Theorem B can be replaced by a larger interval
``$r\in(0,2\sqrt{3}/9)$".

\begin{Thm}\label{CHZ-5}
Let $\omega(t)=t$ and $\phi$ be an analytic self mapping  of the unit disk $\mathbb{D}$. If for some constants
$r\in(0,2\sqrt{3}/9)$, and $\epsilon>0$, for each $w\in\mathbb{D}$, there is a point $z_{w}\in\mathbb{D}$ such that
$$\rho(\phi(z_{w}),w)<r,~\mbox{and}~\frac{1-|z_{w}|^{2}}{1-|\phi(z_{w})|^{2}}|\phi'(z_{w})|>\epsilon,$$
then, there is a constant $C>0$ depending only on $r$ and $\epsilon$
such that
\begin{equation}\label{eq-bounded-below}
\|C_{\phi}(f)\|_{\mathscr{B}_{H_{\omega}}^{1,0}}\geq C\|f\|_{\mathscr{B}_{H_{\omega}}^{1,0}}
\end{equation}
for all $f\in\mathscr{B}_{H_{\omega}}^{1,0}$.
\end{Thm}

In \cite{CPR}, Chen et al. obtained the following result.

\begin{ThmC}{\rm (\cite[Theorem 6]{CPR})}\label{Thm-CPR}
Let $\alpha\in(0,+\infty)$, $\beta\leq\alpha$, $\omega(t)=t$ and
$\phi:\,\mathbb{D}\rightarrow\mathbb{D}$  be an analytic function.
Then the followings are equivalent:
\begin{enumerate}
\item[{\rm (1)}] $C_{\phi}:~\mathscr{B}_{H_{\omega}}^{\alpha,\beta}\cap\mathscr{A}\rightarrow\mathscr{H}^{2}(\mathbb{D})$ is a bounded operator;

\item[{\rm (2)}] $\displaystyle \frac{1}{2\pi}\int_{0}^{2\pi}\int_{0}^{1}
\frac{|\phi'(re^{i\theta})|^{2}}{(1-|\phi(re^{i\theta})|)^{2\alpha}\left(\log\frac{e}{1-|\phi(re^{i\theta})|}\right)^{2\beta}}
(1-r)\,dr\,d\theta<+\infty. $
\end{enumerate}
\end{ThmC}

We improve and generalize Theorem C into the following form.

\begin{Thm}\label{thm-7}
Let $p\in(0,+\infty)$, $\alpha\in(0,+\infty)$, $\beta\leq\alpha$,
$\omega$ be a majorant and $\phi:\,\mathbb{D}\rightarrow\mathbb{D}$
be an analytic function. Then the followings are equivalent:
\begin{enumerate}
\item[{\rm (1)}] $C_{\phi}:~\mathscr{B}_{H_{\omega}}^{\alpha,\beta}\rightarrow\mathscr{H}_{H}^{p}(\mathbb{D})$ is a bounded operator;

\item[{\rm (2)}] $\displaystyle \frac{1}{2\pi}\int_{0}^{2\pi}\left(\int_{0}^{1}
\frac{|\phi'(re^{i\theta})|^{2}(1-r)}{\omega^{2}\left(\big(1-|\phi(re^{i\theta})|^{2}\big)^{\alpha}
\big(\log \frac{e}{1-|\phi(re^{i\theta})|^{2}}\big)^{\beta}\right)}
\,dr\right)^{\frac{p}{2}}\,d\theta<+\infty;
$

\item[{\rm (3)}] $C_{\phi}:~\mathscr{B}_{H_{\omega}}^{\alpha,\beta}\rightarrow\mathscr{H}_{H}^{p}(\mathbb{D})$ is compact.
\end{enumerate}
\end{Thm}

In the other direction, Madigan \cite[Theorem 3.20]{Ma} proved that, for $p\in[1,+\infty)$,
$$C_{\phi}:~\mathscr{H}^{p}(\mathbb{D})\rightarrow\mathscr{B}_{H_{\omega}}^{1,0}\cap\mathscr{A}$$ is a bounded operator
if and only if

\[
\sup_{z\in \mathbb{D}}\left\{\frac{|\phi'(z)|(1-|z|^{2})}{\big(1-|\phi(z)|^{2}\big)^{1+\frac{1}{p}}}\right\}<+\infty,
\]
where $\omega(t)=t$ and $\phi:\,\mathbb{D}\rightarrow\mathbb{D}$
is an analytic function.


In the following, P\'erez-Gonz\'alez and Xiao
generalized \cite[Theorem 3.20]{Ma}
to $p\in (0,+\infty)$
and also obtained the compactness characterization.

\begin{ThmD}{\rm (\cite[Theorem 3.1]{PX} )}\label{Thm-PX}
Let $\omega(t)=t$, $p\in (0,+\infty)$ and
$\phi:\,\mathbb{D}\rightarrow\mathbb{D}$  be an analytic function.
Then
\begin{enumerate}
\item[{\rm (1)}] $C_{\phi}:~\mathscr{H}^{p}(\mathbb{D})\rightarrow\mathscr{B}_{H_{\omega}}^{1,0}\cap\mathscr{A}$ is a bounded operator
if and only if
\[
\sup_{z\in \mathbb{D}}\left\{\frac{|\phi'(z)|(1-|z|^{2})}{\big(1-|\phi(z)|^{2}\big)^{1+\frac{1}{p}}}\right\}<+\infty;
\]
\item[{\rm (2)}] $C_{\phi}:~\mathscr{H}^{p}(\mathbb{D})\rightarrow\mathscr{B}_{H_{\omega}}^{1,0}\cap\mathscr{A}$ is a compact operator
if and only if
\[
\lim_{\phi(z) \to
\mathbb{T}}\frac{|\phi'(z)|(1-|z|^{2})}{\big(1-|\phi(z)|^{2}\big)^{1+\frac{1}{p}}}=0.
\]
\end{enumerate}
\end{ThmD}


By analogy with Theorem D, we improve and generalize \cite[Theorem 3.20]{Ma} into the following form.

\begin{Thm}\label{thm-8}
Let $p\in (1,+\infty)$, $\omega$ be a majorant with
$\lim_{t\rightarrow0^{+}}(\omega(t)/t)<+\infty$ and
$\phi:\,\mathbb{D}\rightarrow\mathbb{D}$ be an analytic function.
Assume that $\alpha=1$ and $\beta\leq 0$ or $\alpha>1$ and $\beta
\in \mathbb{R}$. Then
\begin{enumerate}
\item[{\rm (1)}] $C_{\phi}:~\mathscr{H}_{H}^{p}(\mathbb{D})\rightarrow\mathscr{B}_{H_{\omega}}^{\alpha, \beta}$ is a bounded operator
if and only if
\be\label{PX2-bounded}
\sup_{z\in \mathbb{D}}\left\{\frac{|\phi'(z)|\omega\left(\left(1-|z|^{2}\right)^{\alpha}
\left(\log \frac{e}{1-|z|^{2}}\right)^{\beta}\right)}{\big(1-|\phi(z)|^{2}\big)^{1+\frac{1}{p}}}\right\}<+\infty;
\ee
\item[{\rm (2)}] $C_{\phi}:~\mathscr{H}_{H}^{p}(\mathbb{D})\rightarrow\mathscr{B}_{H_{\omega}}^{\alpha, \beta}$ is a compact operator
if and only if \be\label{PX2-compact} \lim_{\phi(z) \to
\mathbb{T}}\frac{|\phi'(z)|\omega\left(\left(1-|z|^{2}\right)^{\alpha}
\left(\log
\frac{e}{1-|z|^{2}}\right)^{\beta}\right)}{\big(1-|\phi(z)|^{2}\big)^{1+\frac{1}{p}}}=0.
\ee
\end{enumerate}
\end{Thm}

The proofs of Theorems \ref{thm-1CH}, \ref{CHZ-5}, \ref{thm-7},
\ref{thm-8}  will be presented in Section \ref{sec3}.


\section{Proofs of the main results}\label{sec3}

Let $n=1$ in  \cite[Lemma 1.1]{CPW-2012} and \cite[Theorem 1.2]{CPW-2012}, respectively. Then we obtain the following  results.

\begin{LemE}\label{L-A}
For $x\in[0,1]$, let
$$\psi(x)=\sqrt{1+2\alpha}\left(\frac{1+2\alpha}{2\alpha}\right)^{\alpha}x(1-x^{2})^{\alpha}$$
and $a_{0}(\alpha)=1/\sqrt{1+2\alpha}$, where $\alpha\in(0,+\infty)$ is a constant.
Then $\psi$ is increasing in $[0,a_{0}(\alpha)]$, decreasing in $[a_{0}(\alpha),1]$ and $\psi(a_{0}(\alpha))=1.$
\end{LemE}

\begin{ThmF}\label{Thm-M}
Suppose $\omega(t)=t$ and $f\in\mathscr{B}_{H_{\omega}}^{\alpha,0}$ is a holomorphic  mapping
 satisfying
$\|f\|_{\mathscr{B}_{H_{\omega},s}^{\alpha,0}}=1$ and $f'(0)=r_{0}\in(0,1]$,
where $\alpha\in(0,+\infty)$ is a constant. Then, for all $z$ with
$|z|\leq\frac{a_{0}(\alpha)+m_{\alpha}(r_{0})}{1+a_{0}(\alpha)m_{\alpha}(r_{0})}$,
we have

\be\label{eq-5} {\rm
Re}(f'(z))\geq\frac{r_{0}(m_{\alpha}(r_{0})-|z|)}{m_{\alpha}(r_{0})(1-m_{\alpha}(r_{0})|z|)^{1+2\alpha}},\ee
where $m_{\alpha}(r_{0})$ is the unique real root of the equation
$\psi(x)=r_{0}$ in the interval $[0,a_{0}(\alpha)]$ and, $\psi$ and
$a_{0}(\alpha)$ are defined as in Lemma E. Moreover, for all
$z$ with
$|z|\leq\frac{a_{0}(\alpha)-m_{\alpha}(r_{0})}{1-a_{0}(\alpha)m_{\alpha}(r_{0})}$,
we have
\be\label{eq-6}|f'(z)|\leq\frac{r_{0}(m_{\alpha}(r_{0})+|z|)}{m_{\alpha}(r_{0})(1+m_{\alpha}(r_{0})|z|)^{1+2\alpha}}.\ee
Furthermore, the estimates of {\rm (\ref{eq-5})} and {\rm
(\ref{eq-6})} are sharp.
\end{ThmF}

\begin{Rem}
\label{Remark-N}
Theorem F
also holds by replacing the assumption
$\|f\|_{\mathscr{B}_{H_{\omega},s}^{\alpha,0}}=1$
by
$\|f\|_{\mathscr{B}_{H_{\omega},s}^{\alpha,0}}
\leq 1$ as in Bonk, Minda and Yanagihara \cite{BMY97}.
\end{Rem}

\subsection{The proof of Theorem \ref{thm-1CH}} Since $\mathbb{D}$ is a simply connected domain, we see that $f$ admits the canonical decomposition $f = h + \overline{g}$,
where $h$ and $g$ are analytic in $\mathbb{D}$ with $g(0)=0$.
 If $\|f\|_{\mathscr{B}_{H_{\omega},s}^{1,0}}=0$, then it is  easy to conclude (\ref{eq-3}). Without loss of generality, we may assume that $\|f\|_{\mathscr{B}_{H_{\omega},s}^{1,0}}=1$
and $\mathscr{B}_{\omega,f}^{1,0}(z_{1})\leq\mathscr{B}_{\omega,f}^{1,0}(z_{2})$.  For $z\in\mathbb{D}$, let

 $$\varphi(z)=\frac{z_{2}-z}{1-\overline{z}_{2}z},~ w=\varphi^{-1}(z_{1})~\mbox{and}~F(z)=f(\varphi(z))=\mathcal{H}(z)+\overline{\mathcal{G}(z)},$$
 where $\mathcal{H}=h\circ\varphi$ and $\mathcal{G}=g\circ\varphi.$
Then $\|F\|_{\mathscr{B}_{H_{\omega},s}^{1,0}}=\|f\|_{\mathscr{B}_{H_{\omega},s}^{1,0}}=1$ and
$$\rho(z_{1},z_{2})=\rho(\varphi^{-1}(z_{1}),\varphi^{-1}(z_{2}))=\rho(w,0)=|w|.$$
Elementary calculations lead to
\[
\varphi(0)=z_{2}~~\mbox{and}~~|\varphi'(w)|=\frac{1-|\varphi(w)|^{2}}{1-|w|^{2}}=\frac{1-|z_{1}|^{2}}{1-|w|^{2}},
\]
which imply that
\be\label{eq-0.1}\mathscr{B}_{\omega,F}^{1,0}(0)=\mathscr{B}_{\omega,f}^{1,0}(\varphi(0))=\mathscr{B}_{\omega,f}^{1,0}(z_{2})\ee
and
\be\label{eq-0.2}\mathscr{B}_{\omega,F}^{1,0}(w)=(1-|w|^{2})\Lambda_{f}(\varphi(w))|\varphi'(w)|=\mathscr{B}_{\omega,f}^{1,0}(z_{1}).\ee


\noindent {\rm $\mathbf{Case~1.}$} If $\mathscr{B}_{\omega,F}^{1,0}(0)=0$, then it is easy to conclude (\ref{eq-3}).


\noindent {\rm $\mathbf{Case~2.}$} Let $\mathscr{B}_{\omega,F}^{1,0}(0)\neq0$.

In this case, for
$z\in\mathbb{D}$, let
$F_{\theta_{0}}(z)=\mathcal{H}(z)+e^{i\theta_{0}}\mathcal{G}(z)$,
where $\theta_{0}\in[0,2\pi]$ is a real number such that
$|F_{\theta_{0}}'(0)|=\mathscr{B}_{\omega,F}^{1,0}(0).$
Then $\|F_{\theta_{0}}\|_{\mathscr{B}_{H_{\omega},s}^{1,0}}\leq 1$.
Set
$F_{\theta_{0}}'(0)=\beta e^{i\theta}$, where
$\beta=\mathscr{B}_{\omega,F}^{1,0}(0)$. An application of Theorem
F and Remark \ref{Remark-N} to $e^{-i\theta}F_{\theta_{0}}(z)$ gives that, for
$|z|\leq\frac{a_{0}(1)+m_{1}(\beta)}{1+a_{0}(1)m_{1}(\beta)}$,

\be\label{eq-9}\Lambda_{F}(z)\geq{\rm
Re}(e^{-i\theta}F_{\theta_{0}}'(z))\geq\frac{\beta(m_{1}(\beta)-|z|)}{m_{1}(\beta)(1-m_{1}(\beta)|z|)^{3}},\ee
where
\be\label{eq-0.5}\frac{3\sqrt{3}}{2}m_{1}(\beta)(1-m_{1}^{2}(\beta))=\beta. \ee


\noindent {\rm $\mathbf{Subcase~2.1}.$} Suppose that $|w|\leq m_{1}(\beta)$.


Since $$|w|\leq m_{1}(\beta)\leq\frac{a_{0}(1)+m_{1}(\beta)}{1+a_{0}(1)m_{1}(\beta)}$$
and $$1-|w|^{2}\geq1-m_{1}(\beta)|w|\geq\left(1-m_{1}(\beta)|w|\right)^{3},$$
by (\ref{eq-9}),
we see that

\beq\label{eq-0.4} \beta-\mathscr{B}_{\omega,F}^{1,0}(w)&\leq&
\beta-\frac{(1-|w|^{2})\beta(m_{1}(\beta)-|w|)}{m_{1}(\beta)(1-m_{1}(\beta)|w|)^{3}}\\
\nonumber
&=&\frac{\beta}{m_{1}(\beta)}\left(m_{1}(\beta)-\frac{(1-|w|^{2})(m_{1}(\beta)-|w|)}{(1-m_{1}(\beta)|w|)^{3}}\right)\\
\nonumber
&\leq&\frac{\beta}{m_{1}(\beta)}\left(m_{1}(\beta)-(m_{1}(\beta)-|w|)\right)\\
\nonumber &=&\frac{\beta}{m_{1}(\beta)}|w|. \eeq



 It follows from (\ref{eq-0.1}), (\ref{eq-0.2}),   (\ref{eq-0.5}) and (\ref{eq-0.4}) that

\begin{eqnarray*}
\mathscr{B}_{\omega,f}^{1,0}(z_{2})-\mathscr{B}_{\omega,f}^{1,0}(z_{1})&=&
\mathscr{B}_{\omega,F}^{1,0}(0)-\mathscr{B}_{\omega,F}^{1,0}(w)=\beta-(1-|w|^{2})\Lambda_{F}(w)\\ 
&\leq&\frac{\beta}{m_{1}(\beta)}|w|=\frac{3\sqrt{3}}{2}|w|(1-m_{1}^{2}(\beta))\\ 
&\leq&\frac{3\sqrt{3}}{2}|w|.
\end{eqnarray*}

\noindent {\rm $\mathbf{Subcase~2.2}.$} Suppose that $|w|\geq m_{1}(\beta)$.

By (\ref{eq-0.1}), (\ref{eq-0.2}) and (\ref{eq-0.5}), we have

\beqq
\mathscr{B}_{\omega,f}^{1,0}(z_{2})-\mathscr{B}_{\omega,f}^{1,0}(z_{1})&=&\mathscr{B}_{\omega,F}^{1,0}(0)-\mathscr{B}_{\omega,F}^{1,0}(w)=\beta-(1-|w|^{2})\Lambda_{F}(w)\\
&\leq&\beta=\frac{3\sqrt{3}}{2}m_{1}(\beta)(1-m_{1}^{2}(\beta))\\
&\leq&\frac{3\sqrt{3}}{2}(1-m_{1}^{2}(\beta))|w|\\
&\leq&\frac{3\sqrt{3}}{2}|w|.
\eeqq

Combining Subcases 2.1 and 2,2  gives the desired result.

Now we prove the sharpness part. For any $\varepsilon\in(0,
3\sqrt{3}/2]$, let

 \be\label{eq-0.9}
 m^*=\min\left\{a_{0}(1),\sqrt{2\sqrt{3}\varepsilon}/3\right\}\ee
and
\be\label{eq-0.91}\beta=\frac{3\sqrt{3}}{2}m^*\left(1-(m^*)^{2}\right).\ee
Then $m_1(\beta)=m^*$ and
it follows from (\ref{eq-0.9}) that

 \be\label{eq-0.901}\frac{3\sqrt{3}}{2}\left(1-(m_{1}(\beta))^{2}\right)\geq\frac{3\sqrt{3}}{2}-\varepsilon.\ee
 For $z\in\mathbb{D}$, let $$f_{\beta}(z)=\int_{0}^{z}\frac{\beta(m_{1}(\beta)-\xi)}{m_{1}(\beta)(1-m_{1}(\beta)\xi)^{3}}d\xi,$$
which implies that $f_{\beta}'(0)=\beta$. Next we prove
$\|f_{\beta}\|_{\mathscr{B}_{H_{\omega},s}^{1,0}}=1$. For $z\in\mathbb{D}$, let
$$\eta(z)=-\frac{3\sqrt{3}}{4}z^{2}.$$ By Lemma E, we obtain

\be\label{eq-0.92}\|\eta\|_{\mathscr{B}_{H_{\omega},s}^{1,0}}=\sup_{z\in\mathbb{D}}\mathscr{B}_{\omega,\eta}^{1,0}(z)
=\sup_{z\in\mathbb{D}}\left\{\frac{3\sqrt{3}}{2}|z|(1-|z|^{2})\right\}=1.\ee
For $z\in\mathbb{D}$, let
$$\phi_{a}(z)=\frac{m_{1}(\beta)-z}{1-m_{1}(\beta)z},$$
where $a=m_{1}(\beta).$ Then we see that

\beqq
\mathscr{B}_{\omega,f_{\beta}}^{1,0}(z)&=&
\frac{\frac{3\sqrt{3}}{2}(1-|z|^{2})|m_{1}(\beta)-z|\left(1-(m_{1}(\beta))^{2}\right)}{|1-m_{1}(\beta)z|^{3}}\\
&=&(1-|z|^{2})|(\eta(\phi_{a}(z)))'|\\
&=&\mathscr{B}_{\omega,\eta}^{1,0}(\phi_{a}(z)), \eeqq which,
together with  (\ref{eq-0.92}), yields that
$\|f_{\beta}\|_{\mathscr{B}_{H_{\omega},s}^{1,0}}=1$. Let
$z_{1}=m_{1}(\beta)$ and $z_{2}=0$. Then, by (\ref{eq-0.91}), we
have

\beqq
\left|\mathscr{B}_{\omega,f_{\beta}}^{1,0}(z_{2})-\mathscr{B}_{\omega,f_{\beta}}^{1,0}(z_{1})\right|
&=&\beta=\frac{3\sqrt{3}}{2}m_{1}(\beta)\left(1-(m_{1}(\beta))^{2}\right)\\
&=&\frac{3\sqrt{3}}{2}\left(1-(m_{1}(\beta))^{2}\right)
\rho(z_2,z_1), \eeqq which, together with (\ref{eq-0.901}), implies
that
$$\left|\mathscr{B}_{\omega,f_{\beta}}^{1,0}(z_{2})-\mathscr{B}_{\omega,f_{\beta}}^{1,0}(z_{1})\right|\geq
\left(\frac{3\sqrt{3}}{2}-\varepsilon\right)
\rho(z_2,z_1).$$
Therefore, the above inequality shows that the constant $3\sqrt{3}/2$ is sharp. The proof of this theorem is finished.
\qed

We obtain the following lemma
as in \cite[Lemma 1]{C03}
and
\cite[Lemma 2.14]{H-MJOM}.

\begin{Lem}\label{L-O}
$(\ref{eq-bounded-below})$ holds on $\mathscr{B}_{H_{\omega}}^{1,0}$
if and only if there exists a constant $\delta \in (0,1]$ such that
\[
\|C_{\phi}(f)\|_{\mathscr{B}_{H_{\omega},s}^{1,0}}\geq \delta\|f\|_{\mathscr{B}_{H_{\omega},s}^{1,0}},
\quad
\forall f\in \mathscr{B}_{H_{\omega}}^{1,0}.
\]
\end{Lem}

\begin{proof}
Assume that there exists a $\delta \in (0,1]$ such that
$\| f\circ \phi\|_{\mathscr{B}_{H_{\omega},s}^{1,0}}\geq \delta \| f\|_{\mathscr{B}_{H_{\omega},s}^{1,0}}$
for all $f \in {\mathscr{B}_{H_{\omega}}^{1,0}}$.
Let $f\in {\mathscr{B}_{H_{\omega}}^{1,0}}$.
Then,
we have
\begin{align*}
|f(0)|&\leq |f(\phi(0))|+|f(0)-f(\phi(0))|
\\
&\leq
|f(\phi(0))|+\| f\|_{\mathscr{B}_{H_{\omega},s}^{1,0}}
\sigma(0,\phi(0))
\\
&=
|f(\phi(0))|+\frac{\| f\|_{\mathscr{B}_{H_{\omega},s}^{1,0}}}{2}
\log\frac{1+| \phi(0)|}{1-| \phi(0)|}.
\end{align*}
Therefore, we have
\begin{align*}
\| f\|_{\mathscr{B}_{H_{\omega}}^{1,0}}
&\leq
|f(\phi(0))|+\frac{\| f\circ \phi\|_{\mathscr{B}_{H_{\omega},s}^{1,0}}}{\delta}
\left\{ 1+\frac{1}{2}\log\frac{1+| \phi(0)|}{1-| \phi(0)|}\right\}
\\
&\leq
\frac{\| f\circ \phi\|_{\mathscr{B}_{H_{\omega}}^{1,0}}}{\delta}
\left\{ 1+\frac{1}{2}\log\frac{1+| \phi(0)|}{1-| \phi(0)|}\right\}.
\end{align*}
This implies that
$(\ref{eq-bounded-below})$ holds on $\mathscr{B}_{H_{\omega}}^{1,0}$.

The converse can be proved by using arguments similar to those in the proof of
\cite[Lemma 1]{C03}.
This completes the proof.
\end{proof}

\subsection{The proof of Theorem \ref{CHZ-5}}
Without loss of generality, we may assume that $\|f\|_{\mathscr{B}_{H_{\omega},s}^{1,0}}=1$,
where $\omega(t)=t$.
Let $a\in (0,1)$ be arbitrarily fixed.
There is a point $w\in\mathbb{D}$ such that
\be\label{hk-1}\mathscr{B}_{\omega,f}^{1,0}(w)>1-\delta,\ee
where $\delta=a(1-3\sqrt{3}r/2)$ and $r<2\sqrt{3}/9$. From the
assumption, we see that there is a point $z_{w}$ such that

$$\rho(\phi(z_{w}),w)<r<\frac{2\sqrt{3}}{9},~\mbox{and}~\frac{1-|z_{w}|^{2}}{1-|\phi(z_{w})|^{2}}|\phi'(z_{w})|>\epsilon,$$
which, together with (\ref{eq-3}) and  (\ref{hk-1}),  imply that

\beq\label{eq-0.96}
\mathscr{B}_{\omega,f}^{1,0}(\phi(z_{w}))&\geq&\mathscr{B}_{\omega,f}^{1,0}(w)-\frac{3\sqrt{3}}{2}\rho(\phi(z_{w}),w)\\ \nonumber
&\geq&1-\delta-\frac{3\sqrt{3}}{2}r\\ \nonumber
&=&(1-a)\left(1-\frac{3\sqrt{3}}{2}r\right)\\ \nonumber
&>&0.
\eeq
From (\ref{eq-0.96}), we conclude that

$$\|C_{\phi}(f)\|_{\mathscr{B}_{H_{\omega},s}^{1,0}}
\geq\mathscr{B}_{\omega,f}^{1,0}(\phi(z_{w}))\frac{1-|z_{w}|^{2}}{1-|\phi(z_{w})|^{2}}|\phi'(z_{w})|\geq(1-a)\left(1-\frac{3\sqrt{3}}{2}r\right)\epsilon.$$
Since $a\in(0,1)$ is arbitrary.
we have
\[
\|C_{\phi}(f)\|_{\mathscr{B}_{H_{\omega},s}^{1,0}}
\geq
\left(1-\frac{3\sqrt{3}}{2}r\right)\epsilon,
\]
which, together with Lemma \ref{L-O}, gives $(\ref{eq-bounded-below})$.
The proof of this theorem is complete.
\qed


The following result is well-known.

\begin{LemG}{\rm (cf. \cite[Lemma 5]{CPR})}\label{Lemx}
Suppose that $a,~b\in[0,+\infty)$ and $\tau\in(0,+\infty)$. Then
$$(a+b)^{\tau}\leq2^{\max\{\tau-1,0\}}(a^{\tau}+b^{\tau}).$$
\end{LemG}

Consider a weight $\varphi:~[0,1)\rightarrow(0,+\infty)$, that is,
$\varphi$ is a non-decreasing, continuous, unbounded function.
Furthermore, a weight function $\varphi$ is called doubling if there
is a constant $C>1$ such that

$$\varphi(1-s/2)<C\varphi(1-s)$$
for $s\in(0,1]$ (see \cite{AD}).

\begin{Lem}\label{O-lemm-1}
Let  $\alpha\in(0,+\infty)$, $\beta\leq\alpha$ and $\omega$ be a
majorant. Then $\varphi(t)=1/\omega(\chi(t))$ is a doubling
function, where
$$\chi(t)=\left(1-t^{2}\right)^{\alpha}
\left(\log \frac{e}{1-t^{2}}\right)^{\beta}$$ for $t\in[0,1)$.
\end{Lem}

\begin{proof} Since $\omega(x)/x$ is non-increasing for $x>0$ and $\chi(t)$
is decreasing for $t>0$, we see that, for $s\in(0,1]$,

\beq\label{eq-ch-j1}
\frac{\varphi\left(1-\frac{s}{2}\right)}{\varphi(1-s)}&=&\frac{\omega(\chi(1-s))}{\omega\left(\chi\left(1-\frac{s}{2}\right)\right)}=
\frac{\frac{\omega(\chi(1-s))}{\chi(1-s)}}{\frac{\omega\left(\chi\left(1-\frac{s}{2}\right)\right)}{\chi\left(1-\frac{s}{2}\right)}}\cdot\frac{\chi(1-s)}{\chi\left(1-\frac{s}{2}\right)}\\
\nonumber &\leq&\frac{\chi(1-s)}{\chi\left(1-\frac{s}{2}\right)}.
\eeq Elementary computations give
$$\lim_{s\rightarrow0}\frac{\chi(1-s)}{\chi\left(1-\frac{s}{2}\right)}=2^{\alpha}$$
and
$$\lim_{s\rightarrow1^{-}}\frac{\chi(1-s)}{\chi\left(1-\frac{s}{2}\right)}=\left(\frac{4}{3}\right)^{\alpha}\frac{1}{\left(1+\log\frac{4}{3}\right)^{\beta}},$$
which, together with (\ref{eq-ch-j1}) and the continuity of $\chi(1-s)/\chi\left(1-\frac{s}{2}\right)$, implies that there is a
constant $C>1$ such that

$$\varphi\left(1-\frac{s}{2}\right)< C\varphi(1-s). $$
Hence $\varphi$ is a doubling function.
\end{proof}

Denote by $L^{p}(\mathbb{T})$ {\color{red}$(p\in(0,+\infty))$} the set of all measurable functions  $F$
of $\mathbb{T}$ into $\mathbb{C}$ with
$$\|F\|_{L^{p}}=\left(\frac{1}{2\pi}\int_{0}^{2\pi}|F(e^{i\theta})|^{p}d\theta\right)^{\frac{1}{p}}<+\infty.$$



Given $f\in \mathscr{A}$, the Littlewood-Paley
$\mathscr{G}$-function is defined as follows

$$\mathscr{G}(f)(\zeta)=\left(\int_{0}^{1}|f'(r\zeta)|^{2}(1-r)dr\right)^{\frac{1}{2}},~\zeta\in\partial\mathbb{D}.$$

It is well known that
\be\label{1.14ch}f\in \mathscr{H}^{p}(\mathbb{D})~~\mbox{if and only
if}~~ \mathscr{G}(f)\in L^{p}(\mathbb{T})\ee for
$p\in(0,+\infty)$ (see \cite{AB,KL,Kw,P-2013}). Moreover,
(\ref{1.14ch}) also can be rewritten in the following form.
 There is a positive constant $C$, depending only on $p$, such that
\be \label{1.19-ch}
\frac{1}{C}\|f\|_{p}^{p}\leq|f(0)|^{p}+\frac{1}{2\pi}\int_{0}^{2\pi}(\mathscr{G}(f)(e^{i\theta}))^{p}d\theta\leq\, C\|f\|_{p}^p
\ee for $p\in(0,+\infty)$ (see \cite{AB,Stei}).



\subsection{The proof of Theorem \ref{thm-7}}
We first prove $(2)\Rightarrow(1)$.
Without loss of generality, let
$f\in
\mathscr{B}_{H_{\omega}}^{\alpha,\beta}$ with
$\|f\|_{\mathscr{B}_{H_{\omega}}^{\alpha,\beta}}\neq0$.
Since $\mathbb{D}$ is a simply
connected domain, we see that $f$ admits the canonical decomposition $f =
h_{1}+\overline{h_{2}}$, where $h_{1}$ and $h_{2}$ are analytic in
$\mathbb{D}$ with $h_{2}(0)=0$. Then
$h_{1},~h_{2}\in\mathscr{B}_{H_{\omega}}^{\alpha,\beta}.$ Let
$$\chi(t)=\left(1-t^{2}\right)^{\alpha}
\left(\log \frac{e}{1-t^{2}}\right)^{\beta}$$ for $t\in[0,1)$.
 For $j=1,2$, elementary calculations lead to

$$|C_{\phi}(h_{j})'(z)|=|h_{j}'(\phi(z))||\phi'(z)|\leq\|f\|_{\mathscr{B}_{H_{\omega}}^{\alpha,\beta}}\frac{|\phi'(z)|}
{\omega\left(\chi(|\phi(z)|)\right)},$$ which
implies that

\beq\label{eq-pp-1}
&&\frac{1}{2\pi}\int_{0}^{2\pi}\left(\mathscr{G}(C_{\phi}(h_{j}))(e^{i\theta})\right)^{p}d\theta\\ \nonumber
&=&\int_{0}^{2\pi}
\left(\int_{0}^{1}|h_{j}'(\phi(re^{i\theta}))|^{2}|\phi'(re^{i\theta})|^{2}(1-r)dr\right)^{\frac{p}{2}}\frac{d\theta}{2\pi}\\ \nonumber
&\leq&\|f\|_{\mathscr{B}_{H_{\omega}}^{\alpha,\beta}}^{p}\int_{0}^{2\pi}\left(\int_{0}^{1}\frac{|\phi'(re^{i\theta})|^{2}(1-r)dr}
{\omega^{2}\left(\chi(|\phi(re^{i\theta})|)\right)}\right)^{\frac{p}{2}}\frac{d\theta}{2\pi}\\ \nonumber
 &<&+\infty, \eeq where $\theta\in[0,2\pi]$ and
$\zeta=e^{i\theta}.$ It follows from (\ref{1.14ch}) and
(\ref{eq-pp-1}) that
 $C_{\phi}(h_{j})\in
\mathscr{H}^{p}(\mathbb{D})$, where $j=1,2$.

Next, we will show that $C_{\phi}$ is  a bounded operator. Since $\|f\|_{\mathscr{B}_{H_{\omega}}^{\alpha,\beta}}\neq0$,
we see that at least one of the two  functions $h_{1}$ and $h_{2}$ is not constant. Without loss of generality, we can assume
$h_{1}$ and $h_{2}$ are not constant functions.

Since
\begin{eqnarray*}
|C_{\phi}(h_{j})(0)| &\leq& |h_j(0)|+|h_j(\phi(0))-h_j(0)|
\\
& \leq&
 |h_j(0)|+\| h_j\|_{\mathscr{B}_{H_{\omega}}^{\alpha,\beta}}\frac{|\phi(0)|}{\omega\left(\chi(|\phi(0)|)\right)},
\end{eqnarray*}
by using (\ref{1.19-ch}), we see that there is a positive constant $C$, depending only on $p$ and $|\phi(0)|$, such that

\beqq
\|C_{\phi}(h_{j})\|_{p}^{p}&\leq&\, C\left(|C_{\phi}(h_{j})(0)|^{p}+\frac{1}{2\pi}\int_{0}^{2\pi}(\mathscr{G}(C_{\phi}(h_{j}))(e^{i\theta}))^{p}d\theta\right)\\
&\leq&C\|h_{j}\|_{\mathscr{B}_{H_{\omega}}^{\alpha,\beta}}^{p}\left(1+
\int_{0}^{2\pi}\left(\int_{0}^{1}\frac{|\phi'(re^{i\theta})|^{2}(1-r)dr}
{\omega^{2}\left(\chi(|\phi(re^{i\theta})|)\right)}\right)^{\frac{p}{2}}\frac{d\theta}{2\pi}\right),
\eeqq
which, together with Lemma G, gives that
\beqq\|C_{\phi}(f)\|_{p}^{p}&\leq&2^{\max\{p-1,0\}}\left(\|C_{\phi}(h_{1})\|_{p}^{p}+\|C_{\phi}(h_{2})\|_{p}^{p}\right)\\
&\leq&2^{\max\{p-1,0\}}CC^{\ast}\big(\|h_{1}\|_{\mathscr{B}_{H_{\omega}}^{\alpha,\beta}}^{p}+\|h_{2}\|_{\mathscr{B}_{H_{\omega}}^{\alpha,\beta}}^{p}\big)\\
&\leq&2^{1+\max\{p-1,0\}}\|f\|^{p}_{\mathscr{B}_{H_{\omega}}^{\alpha,\beta}}CC^{\ast},
\eeqq
where $$C^{\ast}=1+\int_{0}^{2\pi}\left(\int_{0}^{1}\frac{|\phi'(re^{i\theta})|^{2}(1-r)dr}
{\omega^{2}\left(\chi(|\phi(re^{i\theta})|)\right)}\right)^{\frac{p}{2}}\frac{d\theta}{2\pi}<+\infty.$$
Consequently, $C_{\phi}:~\mathscr{B}_{H_{\omega}}^{\alpha,\beta}\rightarrow\mathscr{H}_{H}^{p}(\mathbb{D})$ is a bounded operator.

Next, we prove $(1)\Rightarrow(2)$. By Lemma \ref{O-lemm-1} and \cite[Lemma 1]{AD}, we see that there are two
functions $f_{1},~f_{2}\in\mathscr{B}_{H_{\omega}}^{\alpha,\beta}\cap\mathscr{A}$ such that, for $z\in\mathbb{D}$,

\[
|f_{1}'(z)|+|f_{2}'(z)|\geq\frac{1}{\omega\left(\chi(|z|)\right)},
\]
which, together with Lemma G, implies that

\beqq
+\infty&>&\frac{1}{2\pi}\int_{0}^{2\pi}\left(\mathscr{G}(C_{\phi}(f_{1}))(e^{i\theta})\right)^{p}d\theta
+\frac{1}{2\pi}\int_{0}^{2\pi}\left(\mathscr{G}(C_{\phi}(f_{2}))(e^{i\theta})\right)^{p}d\theta\\
&\geq&\frac{C_{p}}{2\pi}\int_{0}^{2\pi}
\bigg(\int_{0}^{1}|\phi'(re^{i\theta})|^{2}\Big(|f_{1}'(\phi(re^{i\theta}))|\\
&&+|f_{2}'(\phi(re^{i\theta}))|\Big)^{2}(1-r)dr\bigg)^{\frac{p}{2}}d\theta\\
&\geq&\frac{C_{p}}{2\pi}\int_{0}^{2\pi}
\left(\int_{0}^{1}\frac{|\phi'(re^{i\theta})|^{2}(1-r)}{\omega^{2}\left(\chi(|\phi(re^{i\theta})|)\right)}dr\right)^{\frac{p}{2}}d\theta,
\eeqq where $C_{p}=2^{-p/2-\max\{p/2-1,0\}}.$

Now we prove $(1)\Leftrightarrow(3)$. We only need to prove
$(1)\Rightarrow(3)$ because $(3)\Rightarrow(1)$ is obvious. Let
$C_{\phi}:\,\mathscr{B}_{H_{\omega}}^{\alpha,\beta}\rightarrow
\mathscr{H}_{H}^{p}(\mathbb{D})$ be a bounded operator. Then $f\circ
\phi\in \mathscr{H}_{H}^{p}(\mathbb{D})$ for all $f\in
\mathscr{B}_{H_{\omega}}^{\alpha,\beta}$, and
$|\phi(\zeta)|:=\left|\lim_{r\rightarrow1^{-}}\phi(r\zeta)\right|\leq1$
exists  for almost every $\zeta\in\mathbb{T}$. Suppose that $\{
f_{n}\}=\{ h_{n}+\overline{g_{n}}\}$ is a sequence in
$\mathscr{B}_{H_{\omega}}^{\alpha,\beta}$ such that
$\|f_{n}\|_{\mathscr{B}_{H_{\omega}}^{\alpha,\beta}}\leq1,$ where
$h_{n}$ and $g_{n}$ are analytic in $\mathbb{D}$ with $g_{n}(0)=0$.
We are going to prove that $\{ C_{\phi}(f_{n})\}$ has a convergent
subsequence in $\mathscr{H}_{H}^{p}(\mathbb{D})$. Elementary
calculations give that, for $z\in\mathbb{D}$,

\beqq |f_{n}(z)|\leq |f_{n}(0)|+\|f_{n}\|_{\mathscr{B}_{H_{\omega}}^{\alpha,\beta}}\int_{0}^{1}\frac{|z|}{\omega\left(\chi(|z|t)\right)}dt
\leq1+\int_{0}^{1}\frac{|z|}{\omega\left(\chi(|z|t)\right)}dt, \eeqq
which implies that $\{ f_{n}\}$ forms a normal family, so that there exists a subsequence of $\{ f_{n}\}$ that converges uniformly on compact subsets
of $\mathbb{D}$ to a complex-valued harmonic function $f=h+\overline{g}$ (see \cite[p.80]{Du}), where $h$ and $g$ are analytic in $\mathbb{D}$.
 Without loss of generality,
 we may assume that the sequence $\{ f_{n}\}$ itself converges to $f$. Moreover,
 it is not difficult to know that $h_{n}\rightarrow h$ and $g_{n}\rightarrow g$ as $n\rightarrow+\infty$ (cf. \cite[p.82]{Du} ).
Consequently,

\beqq |f(0)|+\Lambda_{f}(z)\omega\left(\chi(|z|)\right)=\lim_{n\rightarrow+\infty}\left\{ |f_n(0)|+\Lambda_{f_{n}}(z)\omega\left(\chi(|z|)\right)\right\}\leq1, \eeqq
which implies that $f\in\mathscr{B}_{H_{\omega}}^{\alpha,\beta}$ with $\|f\|_{\mathscr{B}_{H_{\omega}}^{\alpha,\beta}}\leq1.$
Since $(1)$ implies $(2)$, we see that
\be\label{com-1}\int_{0}^{1}\frac{|\phi'(r\zeta)|^{2}(1-r)}
{\omega^{2}\left(\chi(|\phi(r\zeta)|)\right)}dr<+\infty\ee a.e. $\zeta\in\mathbb{T}$.
By (\ref{com-1}),
\beq\label{com-2}&&\int_{0}^{1}|(h_{n}-h)'\circ\phi(r\zeta)|^{2}|\phi'(r\zeta)|^{2}(1-r)dr\\ \nonumber
&\leq&\|f_{n}-f\|^{2}_{\mathscr{B}_{H_{\omega}}^{\alpha,\beta}}\int_{0}^{1}\frac{|\phi'(r\zeta)|^{2}(1-r)}{\omega^{2}\left(\chi(|\phi(r\zeta)|)\right)}dr\\
\nonumber
&\leq&4\int_{0}^{1}\frac{|\phi'(r\zeta)|^{2}(1-r)}{\omega^{2}\left(\chi(|\phi(r\zeta)|)\right)}dr
\eeq and the dominated convergence theorem, we have

\beqq
\lim_{n\rightarrow+\infty}\int_{0}^{1}|(h_{n}-h)'\circ\phi(r\zeta)|^{2}|\phi'(r\zeta)|^{2}(1-r)dr=0
\eeqq
a.e. $\zeta\in\mathbb{T}$.
Consequently, (\ref{com-2}) and the dominated convergence theorem again give that

\be\label{com-3}
\lim_{n\rightarrow+\infty}\frac{1}{2\pi}\int_{\mathbb{T}}\left(\int_{0}^{1}|(h_{n}-h)'\circ\phi(r\zeta)|^{2}
|\phi'(r\zeta)|^{2}(1-r)dr\right)^{\frac{p}{2}}|d\zeta|=0.
\ee
It follows from (\ref{1.19-ch}) and (\ref{com-3}) that
\be\label{com-4}\lim_{n\rightarrow+\infty}\frac{1}{2\pi}\int_{\mathbb{T}}|(h_{n}-h)\circ\phi(\zeta)|^{p}|d\zeta|=0,\ee
which yields that $C_{\phi}(h)\in\mathscr{H}_{H}^{p}(\mathbb{D})$
and $C_{\phi}(h_n) \to C_{\phi}(h)$ in $\mathscr{H}_{H}^{p}(\mathbb{D})$
as $n\rightarrow+\infty$.

By a similar proof process of (\ref{com-4}), we have $C_{\phi}(g)\in\mathscr{H}_{H}^{p}(\mathbb{D})$
and $C_{\phi}(g_n) \to C_{\phi}(g)$ in $\mathscr{H}_{H}^{p}(\mathbb{D})$
as $n\rightarrow+\infty$.
Hence $C_{\phi}(f)=C_{\phi}(h)+\overline{C_{\phi}(g)}\in\mathscr{H}_{H}^{p}(\mathbb{D})$
and $C_{\phi}(f_n) \to C_{\phi}(f)$ in $\mathscr{H}_{H}^{p}(\mathbb{D})$
as $n\rightarrow+\infty$.
This completes the proof.
\qed


\subsection{The proof of Theorem \ref{thm-8}}

(1)
Assume that (\ref{PX2-bounded}) holds.
Let $f\in\mathscr{H}_{H}^{p}(\mathbb{D})$.
Since $\mathbb{D}$
is a simply connected domain, we see that $f$ admits the canonical decomposition
$f = h + \overline{g}$, where $h$ and $g$ are analytic in $\mathbb{D}$ with $g(0)=0$.
It follows from \cite[Theorem 2.1]{K-2019} that $h,~g\in\mathscr{H}^{p}(\mathbb{D})$. By Cauchy's integral formula,
we have $$F(z)=\frac{1}{2\pi\,i}\int_{|\zeta|=1}\frac{H(\zeta)}{\zeta-z}d\zeta+
\overline{\frac{1}{2\pi\,i}\int_{|\zeta|=1}\frac{G(\zeta)}{\zeta-z}d\zeta},~z\in\mathbb{D},$$
where $F(z)=f(rz)$, $H(z)=h(rz)$ and $G(z)=g(rz)$ for $r\in[0,1)$. Elementary calculations lead to
$$F_{z}(z)=\frac{1}{2\pi\,i}\int_{|\zeta|=1}\frac{H(\zeta)}{(\zeta-z)^{2}}d\zeta~\mbox{and}~
\overline{F_{\overline{z}}(z)}=\frac{1}{2\pi\,i}\int_{|\zeta|=1}\frac{G(\zeta)}{(\zeta-z)^{2}}d\zeta.$$
Since $p\in (1,\infty)$, using Jensen's inequality, we obtain
\beqq
|F_{z}(z)|^{p}&\leq&\frac{1}{(1-|z|^{2})^{p}}
\left(\frac{1}{2\pi}\int_{|\zeta|=1}\frac{(1-|z|^{2})}{|\zeta-z|^{2}}|H(\zeta)||d\zeta|\right)^{p}\\
&\leq&\frac{1}{(1-|z|^{2})^{p}}\frac{1}{2\pi}\int_{|\zeta|=1}\frac{(1-|z|^{2})}{|\zeta-z|^{2}}|H(\zeta)|^{p}|d\zeta|\\
&\leq&\frac{(1+|z|)^{2}}{(1-|z|^{2})^{p+1}}\frac{1}{2\pi}\int_{|\zeta|=1}|H(\zeta)|^{p}|d\zeta|,
\eeqq which, together with letting $r\rightarrow1^{-}$, implies that
\be\label{eq-chz-p1}|h'(z)|^{p}\leq\frac{4}{(1-|z|^{2})^{p+1}}\|h\|_{p}^{p}.
\ee Similarly, we conclude
\be\label{eq-chz-p2}|g'(z)|^{p}\leq\frac{4}{(1-|z|^{2})^{p+1}}\|g\|_{p}^{p}.
\ee By (\ref{eq-chz-p1}), (\ref{eq-chz-p2}) and \cite[Theorem 2.1]{K-2019}, we have \be \label{growth-Lambda}
\Lambda_f(z)=|h'(z)|+|g'(z)|\leq \frac{4^{\frac{1}{p}}(\| h\|_p+\|
g\|_p)}{(1-|z|^2)^{1+\frac{1}{p}}} \leq
C\frac{1}{(1-|z|^2)^{1+\frac{1}{p}}}\| f\|_p, \ee where $C$ is a
constant which depends only on $p$. Combining this inequality with
(\ref{PX2-bounded}) implies that
$C_{\phi}:~\mathscr{H}_{H}^{p}(\mathbb{D})\rightarrow\mathscr{B}_{H_{\omega}}^{\alpha,
\beta}$ is bounded.

Conversely, assume that
$C_{\phi}:~\mathscr{H}_{H}^{p}(\mathbb{D})\rightarrow\mathscr{B}_{H_{\omega}}^{\alpha, \beta}$ is bounded.
For $a\in \mathbb{D}$,
let
\[
f_a(z)=\left(\frac{1-|\phi(a)|^2}{(1-\overline{\phi(a)}z)^2}\right)^{\frac{1}{p}},
\quad
z\in \mathbb{D}.
\]
Then $\sup_{a\in\mathbb{D}}\| f_a\|_p<+\infty$
and the boundedness of $C_{\phi}$
implies that
\[
\sup_{a\in \mathbb{D}}\left\{\frac{|\phi(a)| |\phi'(a)|\omega\left(\left(1-|a|^{2}\right)^{\alpha}
\left(\log \frac{e}{1-|a|^{2}}\right)^{\beta}\right)}{\big(1-|\phi(a)|^{2}\big)^{1+\frac{1}{p}}}\right\}<+\infty.
\]
If (\ref{PX2-bounded}) does not hold,
then there exists a sequence $\{ a_n\}$ in $\mathbb{D}$ such that
\[
\lim_{n\to +\infty}\frac{ |\phi'(a_n)|\omega\left(\left(1-|a_n|^{2}\right)^{\alpha}
\left(\log \frac{e}{1-|a_n|^{2}}\right)^{\beta}\right)}{\big(1-|\phi(a_n)|^{2}\big)^{1+\frac{1}{p}}}=+\infty
\]
and $\phi(a_n)\to 0$ as $n\to +\infty$.
Since $\omega$ is a majorant with
\[
\lim_{t\rightarrow0^{+}}(\omega(t)/t)<+\infty,
\]
this implies that
\[
\lim_{n\to +\infty}\frac{ |\phi'(a_n)|\left(1-|a_n|^{2}\right)^{\alpha}
\left(\log \frac{e}{1-|a_n|^{2}}\right)^{\beta}}{\big(1-|\phi(a_n)|^{2}\big)^{1+\frac{1}{p}}}=+\infty.
\]
Since $\phi(a_n) \to 0$ as  $n\to +\infty$ and $|\phi'(a_n)|(1-|a_n|^2)\leq 1$ for all $n\geq 1$,
this is a contradiction
when $\alpha=1$ and $\beta\leq 0$ or $\alpha>1$ and $\beta \in \mathbb{R}$.
Thus, (\ref{PX2-bounded}) holds.

(2)
Assume that (\ref{PX2-compact}) holds.
Then, using that $\omega$ is a majorant with
\[
\lim_{t\rightarrow0^{+}}(\omega(t)/t)<+\infty,
\]
it follows that
$C_{\phi}:~\mathscr{H}_{H}^{p}(\mathbb{D})\rightarrow\mathscr{B}_{H_{\omega}}^{\alpha, \beta}$ is bounded.
We will show that if $\{ f_n\}$ is a bounded sequence in $\mathscr{H}_{H}^{p}(\mathbb{D})$
which converges to $0$ uniformly on compact subsets of $\mathbb{D}$,
then $C_{\phi}(f_n) \to 0$ in $\mathscr{B}_{H_{\omega}}^{\alpha, \beta}$ as $n\to +\infty$.
We may assume that $\| f_n\|_p\leq 1$ for all $n\geq 1$.
Let $\varepsilon>0$ be fixed.
Then there exists an $r\in (0,1)$
such that
\be\label{chz-p3}
\frac{|\phi'(z)|\omega\left(\left(1-|z|^{2}\right)^{\alpha}
\left(\log \frac{e}{1-|z|^{2}}\right)^{\beta}\right)}{\big(1-|\phi(z)|^{2}\big)^{1+\frac{1}{p}}}<\frac{\varepsilon}{C}
\ee
for all $z\in \mathbb{D}$ with $|\phi(z)|>r$,
where $C$ is the constant in (\ref{growth-Lambda}).
Combining 
(\ref{growth-Lambda}) and (\ref{chz-p3}) gives
\[
\Lambda_{C_{\phi}(f_n)}(z)\omega\left(\left(1-|z|^{2}\right)^{\alpha}
\left(\log \frac{e}{1-|z|^{2}}\right)^{\beta}\right)
<
\varepsilon
\]
for all $z\in \mathbb{D}$ with $|\phi(z)|>r$ and for all $n\geq 1$.
Since $\{ f_n\}$ converges to $0$ uniformly on compact subsets of $\mathbb{D}$, by using an argument similar to that in (1), we see that
\[
\Lambda_{C_{\phi}(f_n)}(z)\omega\left(\left(1-|z|^{2}\right)^{\alpha}
\left(\log \frac{e}{1-|z|^{2}}\right)^{\beta}\right)
\to 0,
\quad
\mbox{as }
n\to +\infty
\]
uniformly on $|\phi(z)|\leq r$.
Also, $C_{\phi}(f_n)(0) \to 0$ as $n\to +\infty$.
Thus, $C_{\phi}(f_n) \to 0$ in $\mathscr{B}_{H_{\omega}}^{\alpha, \beta}$ as $n\to +\infty$.

Conversely, assume that
$C_{\phi}:~\mathscr{H}_{H}^{p}(\mathbb{D})\rightarrow\mathscr{B}_{H_{\omega}}^{\alpha, \beta}$ is compact.
If (\ref{PX2-compact}) does not hold, then
there exist a constant $\varepsilon_0>0$ and a sequence $\{ a_n\}$ in $\mathbb{D}$
such that $|\phi(a_n)| \to 1$  as $n\to +\infty$ and
\[
\frac{|\phi'(a_n)|\omega\left(\left(1-|a_n|^{2}\right)^{\alpha}
\left(\log \frac{e}{1-|a_n|^{2}}\right)^{\beta}\right)}{\big(1-|\phi(a_n)|^{2}\big)^{1+\frac{1}{p}}}
>\varepsilon_0
\]
for all $n\geq 1$.
We may assume that $a_n\to \mathbb{T}$ and $\phi(a_n) \to b\in \mathbb{T}$
as $n\to +\infty$.
Let
\[
f_n(z)=\left(\frac{1-|\phi(a_n)|^2}{(1-\overline{\phi(a_n)}z)^2}\right)^{\frac{1}{p}},
\quad
z\in \mathbb{D}.
\]
Then $\{ f_n\}$ is a bounded sequence in $\mathscr{H}_{H}^{p}(\mathbb{D})$
which converges to $0$ uniformly on compact subsets of $\mathbb{D}$
as $n\to +\infty$.
On the other hand,
\begin{eqnarray*}
\| C_{\phi}(f_n)\|_{\mathscr{B}_{H_{\omega,p}}^{\alpha,\beta}}
&\geq&
|\phi(a_n)|\frac{|\phi'(a_n)|\omega\left(\left(1-|a_n|^{2}\right)^{\alpha}
\left(\log \frac{e}{1-|a_n|^{2}}\right)^{\beta}\right)}{\big(1-|\phi(a_n)|^{2}\big)^{1+\frac{1}{p}}}
\\
&>&
|\phi(a_n)| \varepsilon_0
\end{eqnarray*}
for all $n\geq 1$.
This contradicts with the compactness of $C_{\phi}$.
Thus, (\ref{PX2-compact}) holds.
This completes the proof.
\qed



\section{Acknowledgments}
 The research of the first author was partly supported by the National Science
Foundation of China (grant no. 12071116), the Hunan Natural Science
outstanding youth fund project, the Key Projects of Hunan Provincial
Department of Education (grant no. 21A0429);
 the Double First-Class University Project of Hunan Province
(Xiangjiaotong [2018]469),  the Science and Technology Plan Project of Hunan
Province (2016TP1020), and the Discipline Special Research Projects of Hengyang Normal University (XKZX21002);
The research of the second author was partly supported by
JSPS KAKENHI Grant
Number JP22K03363;
The research of the third author was partly supported by the National Science
Foundation of China (grant nos.
 11501220 and 11971182), and Fujian Natural Science Foundation (grant nos. 2021J01304 and 2019J0101).


\begin{thebibliography}{1}

\bibitem{AD}  E. Abakumov and  E. Doubtsov,
Reverse estimates in growth spaces,
\textit{Math. Z.} {\bf 271}(2012), 399--413.

\bibitem{AB}  P. Ahern and J. Bruna,
Maximal and area integral characterization of Hardy-Sobolev spaces in the unit ball of $\mathbb{C}^{n}$,
\textit{Rev. Mat. Iberoam.} {\bf 4}(1988), 123--153.

\bibitem{BMY97}
M. Bonk, D. Minda and H. Yanagihara,
Distortion theorem for Bloch functions,
\textit{Pacific J. Math.}
{\bf 179}  (1997),  241--262.

\bibitem{C03}
H. Chen,
Boundedness from below of composition operators on the Bloch space.
\textit{Science in China}
\textbf{46}  (2003), 838--846.

\bibitem{CK}  S. L. Chen and D. Kalaj,
Lipschitz continuity of Bloch type mappings with respect to Bergman
metric, \textit{Ann. Acad. Sci. Fenn. Math.} {\bf 43} (2018),
239--246.




\bibitem{CPR}  S. L. Chen,  S. Ponnusamy, and A. Rasila,
On characterizations of Bloch-type, Hardy-type, and Lipschitz-type spaces, \textit{Math. Z.} {\bf 279} (2015), 163--183.


\bibitem{CPW-2012}  S. L. Chen, S. Ponnusamy and X. T. Wang,
Landau-Bloch constants for functions in $\alpha$-Bloch spaces and Hardy spaces, \textit{Complex Anal. Oper. Theory} {\bf 6} (2012),  1025--1036.




\bibitem{CPW-2021}  S. L. Chen,  S. Ponnusamy and X. T. Wang,
Remarks on `Norm estimates of the partial derivatives for harmonic
mappings and harmonic quasiregular mappings',  \textit{J. Geom.
Anal.}   {\bf 31} (2021), 11051--11060.







\bibitem{Co-1989}  F. Colonna,
The Bloch constant of bounded harmonic mappings, \textit{Indiana Univ. Math. J.} {\bf 38} (1989),  829--840.


\bibitem{Du1}  P. Duren,
{\it Theory of $H^{p}$ spaces,} 2nd ed., Dover, Mineola, N. Y., 2000.

\bibitem{Du}  P. Duren,
{\it Harmonic mappings in the plane,} Cambridge Univ. Press, 2004.

\bibitem{Dy1} K. M. Dyakonov,
Equivalent norms on Lipschitz-type spaces of holomorphic functions,
\textit{Acta Math.} {\bf 178}(1997), 143--167.



\bibitem{Dy2}  K.~M.~Dyakonov,
Holomorphic functions and quasiconformal mappings with smooth moduli,
\textit{Adv. Math.} {\bf 187}(2004),  146--172.







\bibitem{GYZ}  P. Ghatage, J. Yan and D. Zheng,
Composition operators with closed range on the Bloch space, \textit{Proc. Amer. Math. Soc.} {\bf 129} (2000), 2039--2044.





 \bibitem{GPP}  D. Girela, M. Pavlovi\'{c} and J. A. Pel\'{a}ez,
Carleson measures for the Bloch space, \textit{J. Anal. Math.}, {\bf
100}(2006), 53--81.

 \bibitem{GPPJ}  D. Girela,  J. A. Pel\'{a}ez, F. P. Gonz\'{a}lez and  J. R\"{a}tty\"{a},
Carleson measures for the Bloch space,
\textit{Integr. Equ. Oper. Theory}, {\bf 61}(2008), 511--547.







\bibitem{H}  H. Hamada,
Distortion theorems, Lipschitz continuity and their applications for
Bloch type mappings on bounded symmetric domains in
$\mathbb{C}^{n}$, \textit{Ann. Acad. Sci. Fenn. Math.} {\bf 44}
(2019), 1003--1014.

\bibitem{H-MJOM}
H. Hamada,
Closed range composition operators on
the Bloch space of
bounded symmetric domains,
\textit{Mediterr. J. Math.}
\textbf{17}  (2020), 104.


\bibitem{HO-0}  T. Hosokawa and S. Ohno,
Differences of  composition operators on Bloch space, \textit{J.
Operat. Theor.} {\bf 57} (2007), 229--242.

\bibitem{HO}  T. Hosokawa and S. Ohno,
Differences of weighted composition operators acting from Bloch space to $H^{\infty}$, \textit{Trans. Amer. Math. Soc.} {\bf 363} (2011), 5321--5340.





\bibitem{HRZ}  J. Huang, A. Rasila and J. -F. Zhu,
Lipschitz property of harmonic mappings with respect to the
pseudo-hyperbolic metric, \textit{Anal. Math.}  {\color{red}2022, DOI:10.1007/s10476-022-0132-z}.


\bibitem{K-2019}  D. Kalaj, On Riesz type inequalities for harmonic mappings on the unit disk,
 \textit{Trans. Amer. Math. Soc.} {\bf 372} (2019), 4031--4051.






\bibitem{KL}  S. G. Krantz and S. Y. Li, Area integral characterizations of functions in Hardy spaces on domains in $\mathbb{C}^{n}$,
 \textit{Complex Variables} {\bf 32} (1997), 373--399.


\bibitem{Kw}  E. G. Kwon, Hyperbolic mean growth of bounded holomorphic functions in the ball,
 \textit{Trans. Amer. Math. Soc.} {\bf 355} (2002), 1269--1294.







\bibitem{Ma}  K. Madigan,
Composition operators into Lipschitz type spaces, Thesis, SUNY Albany, 1993.







\bibitem{P}  M. Pavlovi\'c,
On Dyakonov's paper Equivalent norms on Lipschitz-type spaces of
holomorphic functions,
\textit{Acta Math.} {\bf 183}(1999), 141--143.



\bibitem{P-2013}  M. Pavlovi\'c,
On the Littlewood-Paley g-function and Calder\'on's area theorem,
\textit{Expo. Math.} {\bf 31}(2013), 169--195.


\bibitem{PX}  F. P\'erez-Gonz\'alez and J. Xiao, Bloch-Hardy pullbacks,
 \textit{Acta Sci. Math. $($Szeged$)$} {\bf 67} (2001), 709--718.



\bibitem{Sha} J. H. Shapiro, The essential norm of a composition operator, \textit{Ann. Math.} {\bf 125} (1987), 375--404.


\bibitem{Stei} E. Stein,
Some problems in harmonic analysis, \textit{Proceedings of symposium in pure mathematics,}
{\bf 35} (1979), 3--19.



\bibitem{X}  C. J. Xiong,
On the Lipschitz continuity of the dilation of Bloch functions,
\textit{Period. Math. Hung.} {\bf 47} (2003), 233--238.

\bibitem{Zhu}  J.-F. Zhu,
Norm estimates of the partial derivatives for harmonic mappings and
harmonic quasiregular mappings, \textit{J. Geom. Anal.}   {\bf 31}
(2021), 5505--5525.

 \bibitem{Z2}  K. Zhu,
\textit{Operator theory in function spaces}, Marcel Dekker, New
York, 1990.

\bibitem{Z1}  K. Zhu,
\textit{Spaces of holomorphic functions in the unit ball,}
Springer, New York, 2005.


\end{thebibliography}
\end{document}